\documentclass[12pt]{article}

\usepackage[cp1250]{inputenc}
\usepackage{graphicx}
\usepackage{url}
\usepackage[all]{xy}
\usepackage{multicol}
\usepackage{amsthm}
\usepackage{amsmath}
\usepackage{amssymb}
\usepackage{geometry}
\usepackage{pinlabel}
\usepackage{hyperref}
\usepackage{array}

\newtheorem{Theorem}{Theorem}[section]
\newtheorem{lemma}[Theorem]{Lemma}
\newtheorem{proposition}[Theorem]{Proposition}
\newtheorem{corollary}[Theorem]{Corollary}
\newtheorem{definition}[Theorem]{Definition}

\newtheorem{example}[Theorem]{Example}

\newcommand{\ZZ}{\mathbb {Z}}

\newcommand{\CC}{\mathbb {C}}
\newcommand{\Se}{S_{\epsilon}^{3}}

\providecommand{\keywords}[1]
{
  \small	
  \textbf{Keywords :} #1
}

\begin{document}

\title{Algebraic links in lens spaces}

\author{Eva Horvat\\
University of Ljubljana, Faculty of Education \\
Kardeljeva plo\v s\v cad 16, 1000 Ljubljana, Slovenia\\
eva.horvat@pef.uni-lj.si}

%\subjclass[2010]{14J17, 32S05, 32S25, 57M25}

\maketitle

\begin{abstract}
The lens space $L_{p,q}$ is the orbit space of a $\ZZ _{p}$-action on the three sphere. We investigate polynomials of two complex variables that are invariant under this action, and thus define links in $L_{p,q}$. We study properties of these links, and their relationship with the classical algebraic links. We prove that all algebraic links in lens spaces are fibered, and obtain results about their Seifert genus. We find some examples of algebraic knots in $L_{p,q}$, whose lift in the $3$-sphere is a torus link.  
\end{abstract}


\section {Introduction}
\label{sec0}
Classical algebraic links have been studied almost a century ago \cite{BR, BU, KA}. Complex algebraic curves and links of their singularities present a field where complex algebraic geometry meets topology and knot theory. Because of their rich (and at the same time rigid) structure, they might still present a valuable insight into differential geometric structures (such as contact structures) on the ambient 3-manifold. \\

Our interest lies in exploring links in lens spaces, which is a relatively young branch of knot theory. As the classical algebraic links in the 3-sphere represent a special class of links (so-called \textit{iterated torus links}), we would like to know which links play the corresponding role in lens spaces. Do they share the topological properties of the classical algebraic links? Can we find examples of algebraic links in $L_{p,q}$ using our knowledge about their lifts in the 3-sphere?   \\

The lens space $L_{p,q}$ is the orbit space of an action of the cyclic group $G_{p,q}\cong \ZZ _{p}$ on $S^{3}$. We define polynomials $f(x,y)$ of two complex variables, whose corresponding variety $V=f^{-1}(0)$ is invariant under this action. A classical link $\widetilde{K}$ of such a polynomial is the lift of a link $K$ in the lens space $L_{p,q}$, which we call an \textbf{algebraic link}. We show that every algebraic link in $L_{p,q}$ bounds a smooth surface (Proposition \ref{prop2}, Lemma \ref{lemma5}) and is therefore nullhomologous. We investigate the relationship between the number of components of an algebraic link in a lens space and its lift in the 3-sphere, and show that the lift of every algebraic knot in $L_{p,q}$ has $p$-components (Corollary \ref{cor3}). In Theorem \ref{th4}, we show that every algebraic link $K$ in a lens space $L_{p,q}$ defines a fibration $\Psi \colon L_{p,q}\backslash K\to S^{1}$. In Proposition \ref{prop4}, we obtain a formula that relates the Seifert genus of an algebraic knot in $L_{p,q}$ and that of its lift in $S^{3}$. In Corollary \ref{cor5}, we show that every torus link $T(a,b)$ with $\gcd (a,b)=p$ is the lift of an algebraic knot in $L_{p,q}$. \\

The paper is organized as follows. In Section \ref{sec1}, we summarize the background material which will be used in the paper. In Subsection \ref{subs11}, we recall some basic facts about links of complex curve singularities. Subsection \ref{subs12} contains the definition of lens spaces $L_{p,q}$ and recalls their construction using surgery. In Subsection \ref{subs13}, we briefly describe various diagrams of a link $K$ in $L_{p,q}$ and recall the procedure that turns a diagram of $K$ into the diagram of its lift in the 3-sphere. Section \ref{sec2} is the core of the paper. We define $(p,q)$-invariant polynomials of two complex variables and show that such a polynomial defines a link in the lens space $L_{p,q}$. We show that every algebraic link in $L_{p,q}$ is nullhomologous, and obtain implications about the number of its components. Moreover, we show that the space  $L_{p,q}\backslash K$ is a smooth fiber bundle over $S^{1}$, and obtain the relationship between the Seifert genus of $K$ and that of its lift $\widetilde{K}$ in $S^{3}$. In Section \ref{sec3} we determine which algebraic knots/links in $L_{p,q}$ are lifted to torus links in $S^{3}$, and show some examples. We conclude by listing several open problems in Section \ref{sec4}.

\section{Preliminaries}
\label{sec1}

Throughout the paper, we will denote by $\Se =\{(x,y)\in \CC ^{2}\, |\, |x|^{2}+|y|^{2}=\epsilon ^{2}\}$ the 3-sphere with radius $\epsilon $, centered at the origin of $\CC ^{2}$. As usual, $S^{3}_{1}$ will be denoted by $S^{3}$. 

\subsection{The link of a complex curve singularity}
\label{subs11}
Singularities of complex hypersurfaces in $\CC ^{n}$ were explored by Milnor \cite{MIL}. Here we recall some classical results about singular points of complex curves in $\CC ^{2}$.

Let $f\colon \CC ^{2}\to \CC $ be a non-constant polynomial in two complex variables with $f(0,0)=0$. We require that $f$ is either irreducible or it is a product of distinct irreducible polynomials.  Then $f$ defines an algebraic curve $V=f^{-1}(0)$ with a finite number of singular points. Denoting by $r$ the number of local analytic branches of $V$ passing through the origin, the intersection $$K=V\cap S_{\epsilon}^{3}$$ is a smooth compact 1-manifold; an $r$-component link in the 3-sphere $S_{\epsilon }^{3}$ \cite[Corollary 2.9]{MIL}. Moreover, the mapping $\Phi \colon S_{\epsilon }^{3}\backslash K\to S^{1}$, given by $\Phi (x,y)=\frac{f(x,y)}{|f(x,y)|}$, is the projection map of a smooth fiber bundle \cite[Theorem 4.8]{MIL}. 

\begin{definition} A link $K$ in $S^{3}$ is called \textbf{algebraic} if there exists a complex polynomial $f(x,y)$ with $f(0,0)=0$ and some $\epsilon >0$, such that the link $f^{-1}(0)\cap \Se $ is isotopic to $K$. 
\end{definition}

The simplest family of algebraic links consists of torus links. A \textbf{torus link} is a link that lies on the boundary of an unknotted solid torus in $S^{3}$. It is specified by a pair of integers $a$ and $b$ : the torus link $T(a,b)$ winds $a$ times along the meridian and $b$ times along the preferred longitude of the solid torus. The torus link $T(a,b)$ is an algebraic link, associated with the complex polynomial $f(x,y)=x^{a}+y^{b}$. \\

In general, links that arise at isolated singularities of complex curves are the so-called \textbf{ite\-ra\-ted torus links} (tubular links), that were initially studied in \cite{BR}, \cite{KA}, \cite{BU}. To help visualization, the sphere $S_{\epsilon }^{3}$ is usually represented by the union of two solid tori: 
\begin{xalignat}{1}\label{eq1}
& \left \{ (x,y)\in \CC ^{2}\, |\, |x|\leq |y|\textrm{ or }|y|\leq |x|\right \}=T_1 \cup T_2
\end{xalignat}
and by choosing suitable coordinates, we may assume that the link corresponding to $K$ lies within the solid torus $T_{1}$. 

Suppose that a branch of the algebraic curve $V$ has the Puiseux expansion 
\begin{xalignat*}{1}
& y=a_{1}x^{\frac{N_1}{m}}+a_{2}x^{\frac{N_2}{m}}+ a_{3}x^{\frac{N_3}{m}}+\ldots \;,
\end{xalignat*} where $m\leq N_{1}<N_{2}<N_{3}<\ldots $. The above expansion may be rewritten as
\begin{xalignat}{1}\label{eq2}
& y=a_{1}x^{\frac{n_1}{m_1}}+a_{2}x^{\frac{n_2}{m_1 m_2}}+ a_{3}x^{\frac{n_3}{m_1 m_2 m_3}}+\ldots \;,
\end{xalignat} 
where $m_{i}$ and $n_{i}$ are coprime numbers with $m_1\leq n_1\;, \, n_1 m_2 < n_2\;,\, n_2 m_3 < n_3\;,\ldots $. The component of $K$, arising from this branch, is isotopic to the knot $A_{k}$, where $k$ is the smallest integer for which $m_{1}m_{2}\ldots m_{k}=m$, and the knots $A_{i}$ are defined inductively as follows. $A_{1}$ is the torus knot $T(n_{1},m_{1})$ that lies on the boundary of the solid torus $T_1$. For $i\geq 2$, the knot $A_{i}$ lies on the boundary of the regular neighborhood $\nu _{i-1}$ of $A_{i-1}$, winding $n_{i}$ times along the meridian and $m_{i}$ times along the prefered longitude of $\nu _{i-1}$. The knot $A_{k}$ thus defined is sometimes denoted by the symbol $\{(m_{1},n_{1});\, (m_{2},n_{2});\, \ldots ;\, (m_{k},n_{k})\}$ \cite{RE}. 

\begin{figure}[h!]
\labellist
\normalsize \hair 2pt
\endlabellist
\begin{center}
\includegraphics[scale=0.25]{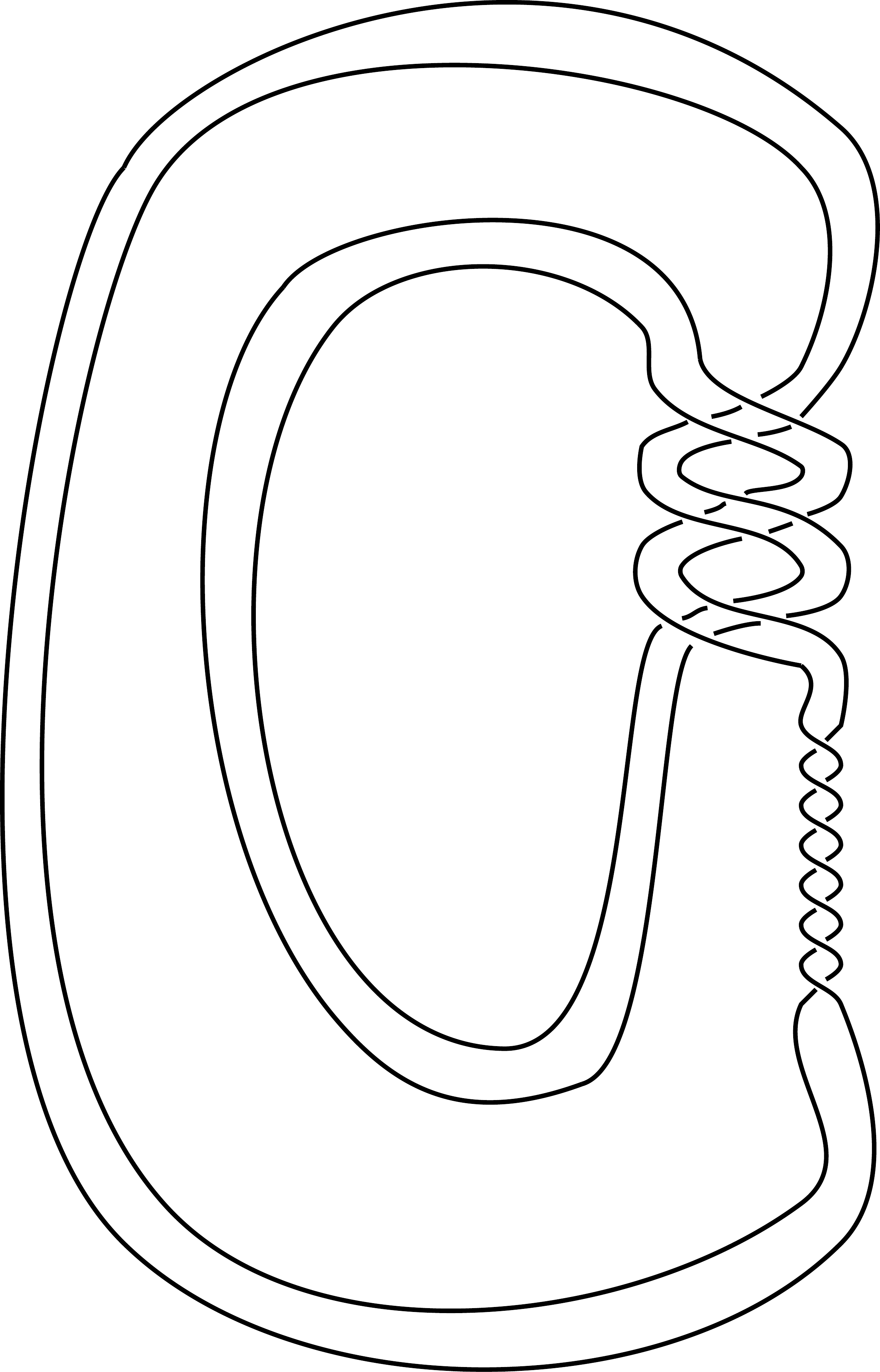}
\caption{An iterated torus knot}
\label{fig0}
\end{center}
\end{figure}

\subsection{Lens spaces}
\label{subs12}
Let $p$ and $q$ be relatively prime integers, and let $\zeta \in \CC $ be a primitive $p^{th}$ root of unity. Then $\zeta $ acts on the 3-sphere $\Se $ by 
\begin{xalignat}{1} \label{eq3}
& \zeta \cdot (x,y)=(\zeta x,\zeta ^{q}y)\;.
\end{xalignat}
Denote by $G_{p,q}$ the cyclic group $\langle \zeta \rangle \cong \ZZ _{p}$,  generated by this action. The orbit space $\Se /\langle \zeta \rangle $ is the lens space $L_{p,q}$. The quotient map $\pi \colon \Se \to L_{p,q}$ is a covering projection. Recall the following well-known result.

\begin{Theorem}\cite[Theorem 7.13]{LEE} \label{th1} Suppose $\widetilde{M}$ is a smooth manifold, and a discrete Lie group $\Gamma $ acts smoothly, freely, and properly on $\widetilde{M}$. Then $\widetilde{M}/\Gamma $ is a topological manifold and has a unique smooth structure such that $\pi \colon \widetilde{M}\to \widetilde{M}/\Gamma $ is a smooth covering map.
\end{Theorem}

Another useful representation of the lens space $L_{p,q}$ is the following. Decompose the 3-sphere $\Se $ into two solid tori $T_{1}$ and $T_{2}$ as in \eqref{eq1}. Denote by $\mu _{i}$ (respectively $\lambda _{i}$) the meridian (respectively longitude) of the solid torus $T_{i}$. Remove $T_{2}$ from $\Se $ and reglue it back along the boundary of $T_{1}$ by a homeomorphism $\phi \colon \partial T_{2}\to \partial T_{1}$, for which $\phi _{*}(\mu _{2})=p\lambda _{1}-q\mu _{1}$. The resulting 3-manifold is precisely the lens space $L_{p,q}$. In other words, $L_{p,q}$ is obtained from the 3-sphere by a $(p,-q)$ surgery on the solid torus $T_{2}$, and the corresponding Kirby diagram of $L_{p,q}$ is an unknot with framing $-\frac{p}{q}$.  
 
\subsection{Links in $L_{p,q}$} 
\label{subs13}
Knots and links in lens spaces were studied in \cite{CMM, GA, MAN2}. We briefly recall some basic notions which will be used in the paper. 

A link $K$ inside a lens space $L_{p,q}$ may be represented in different ways. We may draw $K$ inside the Kirby diagram of $L_{p,q}$ mentioned above, to obtain a so-called mixed link diagram. Since the surgery link is simply an unknot $U$, we may use a projection in which $U$ intersects the plane of the diagram transversely in a single point, which we denote by a dot. In this way, we obtain a \textbf{punctured disk diagram} of $K$. By cutting the plane of the punctured disk diagram along a ray whose endpoint is the dot, we obtain a \textbf{band diagram} of $K$.  

\begin{figure}[h!]
\labellist
\normalsize \hair 2pt
\pinlabel $U$ at 0 40
\pinlabel $K$ at 300 300
\endlabellist
\begin{center}
\includegraphics[scale=0.3]{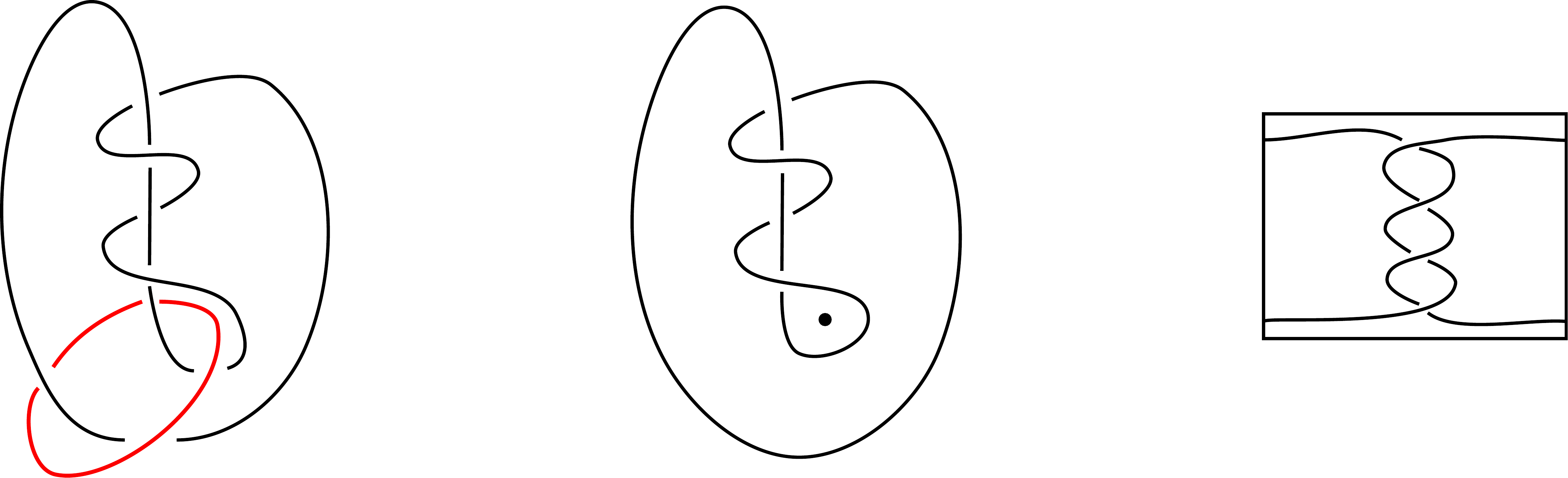}
\caption{The mixed link diagram (left), punctured disk diagram (middle) and the band diagram (right) of a knot $K$ in $L_{p,q}$}
\label{fig1}
\end{center}
\end{figure}

Starting from a band diagram of a link $K$ in $L_{p,q}$, it is possible to obtain a diagram for the lift of $K$ in the 3-sphere \cite{MAN2}. The Garside braid $\Delta _n$ on $n$ strands is given by $$\Delta _n = (\sigma _{n-1}\sigma _{n-2}\ldots \sigma _{1})(\sigma _{n-2}\ldots \sigma _{1})\ldots \sigma _{1}\;,$$ where $\sigma _{i}$ are the standard Artin's generators of the braid group $B_{n}$. 

\begin{proposition}\cite[Proposition 6.4]{MAN2} \label{prop6} Let $K$ be a link in the lens space $L_{p,q}$ and let $B_{K}$ be a band diagram of $K$ with $n$ boundary points. Then a diagram for the lift $\widetilde{K}$ in the 3-sphere $S^{3}$ can be found by juxtaposing $p$ copies of $B_{K}$ and closing them with the braid $\Delta _{n}^{2q}$. 
\end{proposition}

\begin{example} \label{ex1} Consider the knot $K$ in the lens space $L_{3,1}$, which is given by the band diagram on the right of Figure \ref{fig1}. The band diagram has 2 boundary points, and $\Delta _{2}=\sigma _{1}$. By Proposition \ref{prop6}, the lift of $K$ in the 3-sphere is given by the diagram in Figure \ref{fig2}. 
\begin{figure}[h!]
\labellist
\normalsize \hair 2pt
\endlabellist
\begin{center}
\includegraphics[scale=0.25]{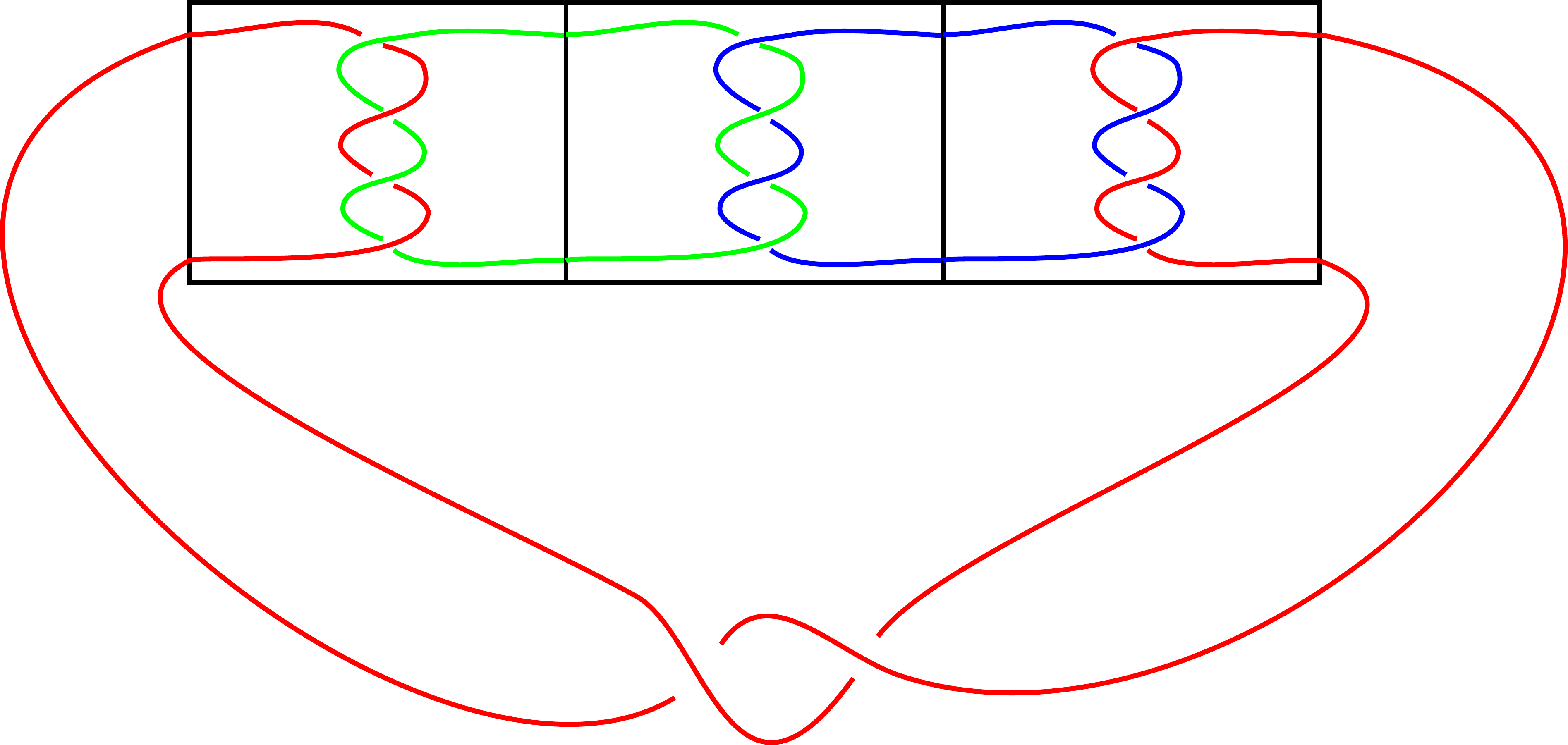}
\caption{The diagram of $\widetilde{K}$, the lift of $K$ in $S^{3}$ (see Example \ref{ex1})}
\label{fig2}
\end{center}
\end{figure}

\end{example}

\section{Algebraic links in lens spaces}
\label{sec2}
We are going to introduce complex polynomials whose singular points define links in lens spaces. Throughout this Section, we denote by $p,q$ two coprime integers, and by $\zeta \in \CC $ a primitive $p^{th}$ root of unity. \\

The radius function $\rho (x,y)=|x|^{2}+|y|^{2}$ induces a lamination of $\CC ^{2}\backslash \{0\}$ into concentric spheres. Denoting by $\tau \colon \CC ^{2}\backslash \{0\}\to S^{3}$ the normalization map, given by $\tau (\mathbf{z})=\frac{\mathbf{z}}{\rho (\mathbf{z})^{\frac{1}{2}}}$, the pair $(\rho ,\tau )$ defines an orientation preserving diffeomorphism $$(\rho ,\tau )\colon \CC ^{2}\backslash \{0\}\to (0,\infty )\times S^{3}\;.$$ Via this diffeomorphism, the 3-sphere $\Se $ is identified by $\{\epsilon \}\times S^{3}$. The action of the Lie group $G_{p,q}=\langle \zeta \rangle $, defined by \eqref{eq3}, may be imposed on $(0,\infty )\times S^{3}$ to yield the orbit space $(0,\infty )\times S^{3}/G_{p,q}\approx (0,\infty )\times L_{p,q}$. Consider complex polynomials whose zero-set is invariant under this action. 

\begin{definition} A polynomial $f(x,y)$ in two complex variables will be called \textbf{$(p,q)$-invariant} if there exists an integer $0\leq k<p$, such that $f(\zeta \cdot (x,y))=\zeta ^{k}f(x,y)$ for any $(x,y)\in \CC ^{2}$. 
\end{definition}

Denoting by $\pi \colon S^{3}\to L_{p,q}$ the quotient projection, we obtain the following diagram:
\begin{displaymath}
\label{eq4}
\xymatrix{\CC^{2}\backslash \{(0,0)\} \ar@{->}[r]^{\quad \quad (\rho ,\tau )\quad \quad } \ar@{->}[dr]_{f} & (0,\infty )\times S^{3} \ar@{->}[r]^{\mathrm{id}\times \pi } & (0,\infty )\times L_{p,q}\\
\quad & \CC & \quad }
\end{displaymath}
Our definition immediately yields:

\begin{proposition} \label{prop1} If $f$ is a $(p,q)$-invariant polynomial, then the algebraic link $\widetilde{K}=f^{-1}(0)\cap \Se $ is the lift of a smooth link in $L_{p,q}$. 
\end{proposition}
\begin{proof} Denote by $V=f^{-1}(0)$ the algebraic curve defined by $f$. Since $f$ is $(p,q)$-invariant, the link $\widetilde{K}=V\cap \Se $ is invariant under the action of $G_{p,q}$. It follows that $\pi ^{-1}(\pi (\widetilde{K}))=\widetilde{K}$. Since $\widetilde{K}$ is a smooth manifold \cite[Corollary 2.9]{MIL}, also $\pi (\widetilde{K})$ is a smooth manifold by Theorem \ref{th1}. 
\end{proof}

\begin{definition} A link $K$ in the lens space $L_{p,q}$ will be called \textbf{algebraic} if there exists a $(p,q)$-invariant polynomial $f$, such that the lift of $K$ in $S^{3}$ is isotopic to the algebraic link $f^{-1}(0)\cap \Se $ for some $\epsilon >0$. 
\end{definition}

\begin{example}\label{ex2} Consider the polynomial $f(x,y)=x^{8}+y^{2}$. The algebraic link $f^{-1}(0)\cap S^{3}$ is the 2-component torus link $T(8,2)$. Let $\zeta \in \CC $ be a primitive third root of unity. Since $f(\zeta x,\zeta y)=\zeta ^{2}f(x,y)$ for any $(x,y)\in \CC ^{2}$, $f$ is $(3,1)$-invariant. It follows by Proposition \ref{prop1} that $T(8,2)$ is the lift of an algebraic link in $L_{3,1}$. Indeed: consider the 2-component link $K$ in $L_{3,1}$, given by the diagrams on Figure \ref{fig3}. Applying Proposition \ref{prop6}, the lift of $K$ in the 3-sphere is given by the diagram on Figure \ref{fig4}, which clearly represents the torus link $T(8,2)$.   

\begin{figure}[h!]
\labellist
\normalsize \hair 2pt
\endlabellist
\begin{center}
\includegraphics[scale=0.25]{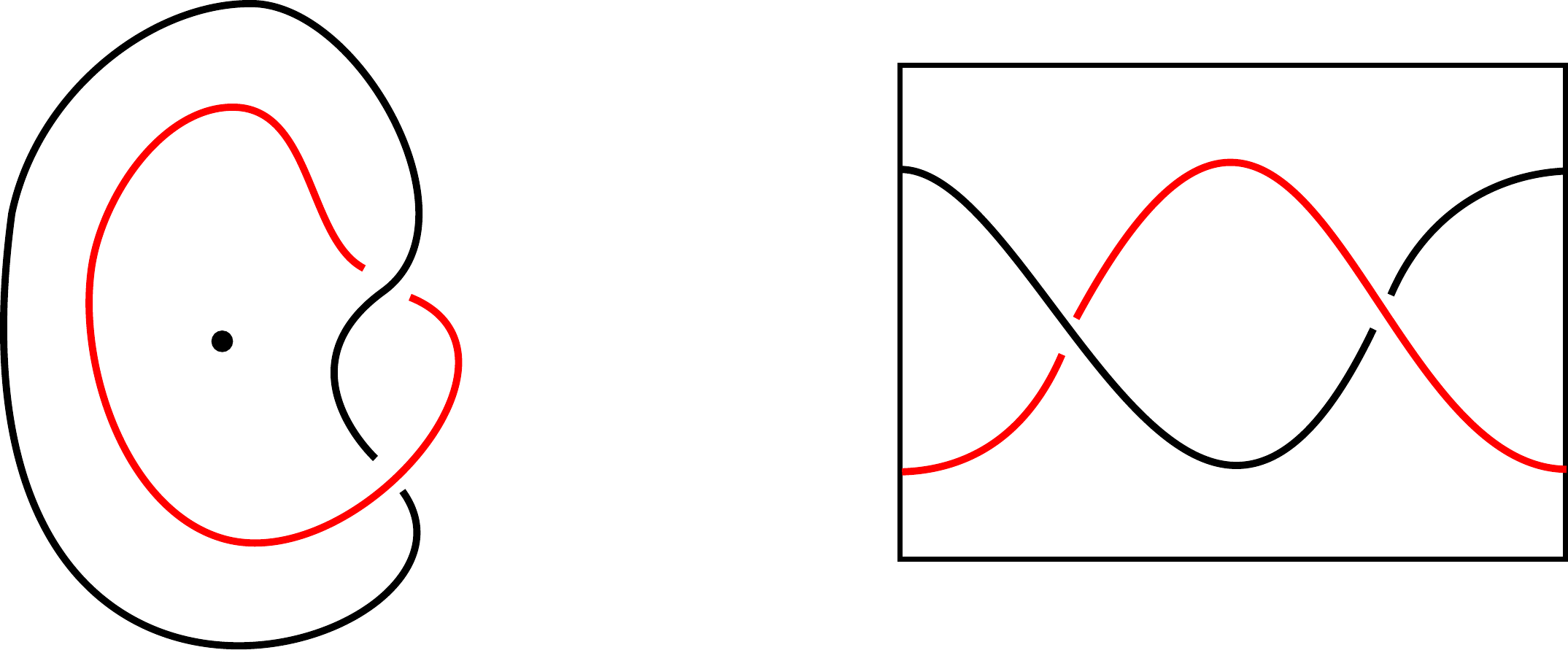}
\caption{The algebraic link in $L_{3,1}$ from Example \ref{ex2} (punctured disk diagram on the left and band diagram on the right)}
\label{fig3}
\end{center}
\end{figure}
\begin{figure}[h!]
\labellist
\normalsize \hair 2pt
\endlabellist
\begin{center}
\includegraphics[scale=0.25]{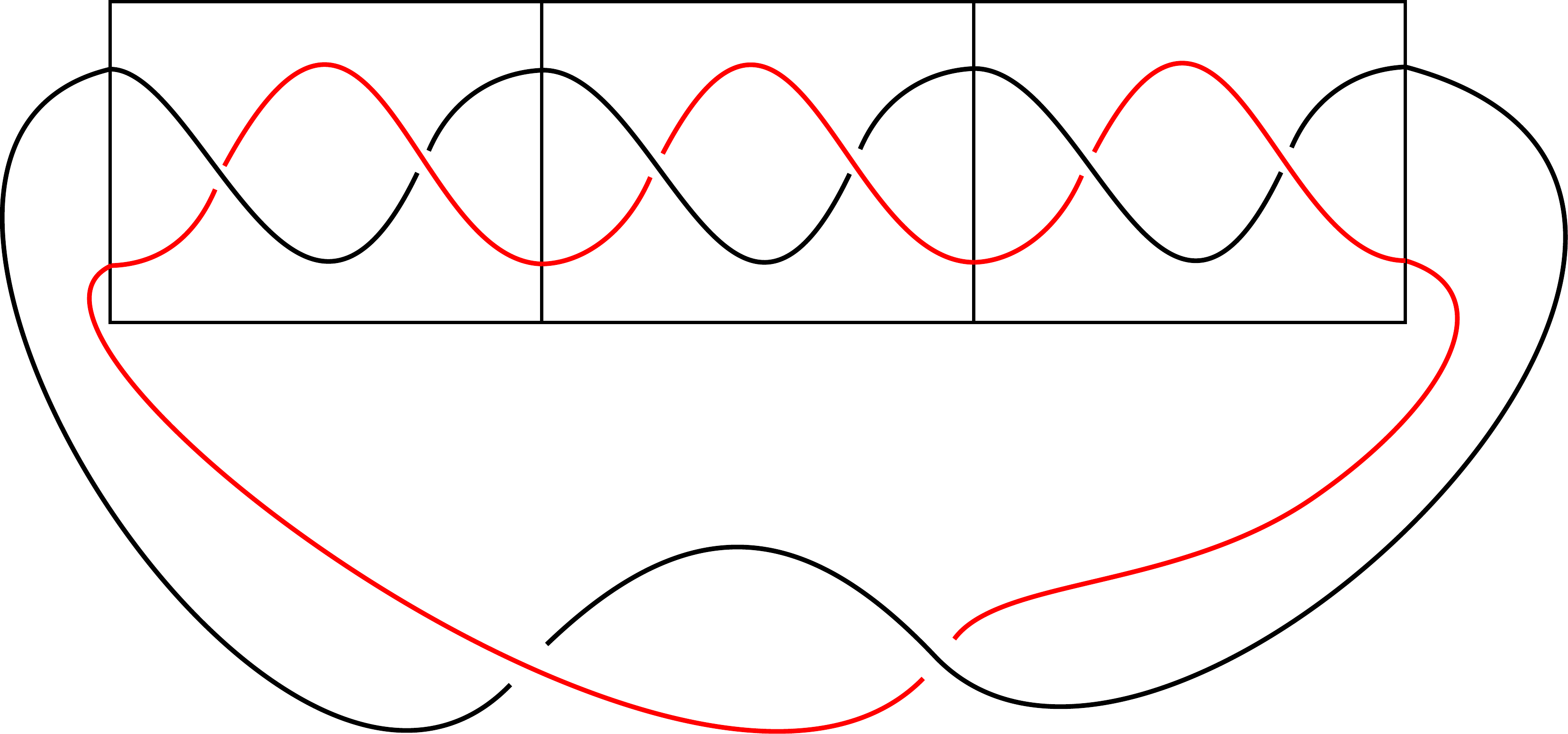}
\caption{The lift in $S^{3}$ of the link in $L_{3,1}$ (see Example \ref{ex2})}
\label{fig4}
\end{center}
\end{figure}
\end{example}

Recall that a link $K$ in a 3-manifold $M$ is called \textbf{fibered} if its complement $M-K$ is the total space of a fiber bundle over $S^{1}$. We would like to show that every algebraic link in a lens space is fibered. 

For the remainder of this Section, let $f(x,y)$ be a $(p,q)$-invariant polynomial that defines an algebraic link $K$ in $L_{p,q}$, whose lift is an algebraic link $\widetilde{K}=f^{-1}(0)\cap \Se $ in the 3-sphere. 

By a classical result of Milnor, every algebraic link in the 3-sphere is fibered. Define a map $\Phi \colon \Se \backslash \widetilde{K}\to S^{1}$ by 
\begin{xalignat}{1}\label{eq5}
& \Phi (x,y)=\frac{f(x,y)}{|f(x,y)|}\;.
\end{xalignat}

\begin{lemma}\cite[Corollary 4.5]{MIL}\label{lemma3} 
Given any non-constant polynomial $f(x,y)$ which vanishes at the origin, there exists an $\epsilon _{0}>0$ such that, for $\epsilon \leq \epsilon _{0}$, the map $\Phi \colon \Se \backslash \widetilde{K}\to S^{1}$ has no critical points at all. 
\end{lemma}

\begin{Theorem}\label{th2} \cite[Theorem 4.8]{MIL} Let $\widetilde{K}$ be an algebraic link in $S^{3}$, associated to a complex polynomial $f$. For $\epsilon \leq \epsilon _{0}$, the space $\Se \backslash \widetilde{K}$ is a smooth fiber bundle over $S^{1}$ with projection mapping $\Phi $. 
\end{Theorem} 

Since $f$ is $(p,q)$-invariant, we have $f(\zeta x,\zeta ^{q}y)=\zeta ^{k}f(x,y)$ for some $0\leq k<p$. Even though the zero-set of the polynomial $f$ is invariant under the action of $G_{p,q}$, its nonzero values are not (unless $k=0$). Denote $\overline{p}=\frac{p}{\gcd (k,p)}$ and define a covering map $\mu \colon S^{1}\to S^{1}$ by $$\mu (z)=z^{\overline{p}}\;.$$ 

\begin{proposition} The map $\Phi \colon \Se \backslash \widetilde{K}\to S^{1}$ induces a well-defined map $\Psi \colon L_{p,q}\backslash K\to S^{1}$, such that the following diagram commutes:
\begin{displaymath}
\label{eq5}
\xymatrix{\Se \backslash \widetilde{K} \ar@{->}[r]^{\pi } \ar@{->}[d]_{\Phi } &  L_{p,q}\backslash K  \ar@{->}[d]^{\Psi }\\
S^{1} \ar@{->}[r]^{\mu } & S^{1}}
\end{displaymath}
\end{proposition}
\begin{proof} Denote $\pi (x,y)=[x,y]$. Following the diagram, the map $\Psi $ is given by $\Psi ([x,y])=(\mu \circ \Phi )(x,y)$. To see that this is well-defined, compute 
\begin{xalignat*}{1}
& (\mu \circ \Phi )(\zeta x,\zeta ^{q}y)=\mu \left (\frac{f(\zeta x,\zeta ^{q}y)}{|f(\zeta x,\zeta ^{q}y)|}\right )=\mu \left (\frac{\zeta ^{k}f(x,y)}{|f(x,y)|}\right )=\mu \left (\frac{f(x,y)}{|f(x,y)|}\right )=(\mu \circ \Phi )(x,y)\;.
\end{xalignat*}
\end{proof}

By \cite[Lemma 6.1]{MIL}, the closure of each fiber of the fibration $\Phi \colon \Se \backslash \widetilde{K}\to S^{1}$ is a Seifert surface of the link $\widetilde{K}$. To obtain a similar statement about the map $\Psi $, we make use of Lemma \ref{lemma3} together with the following result:

\begin{Theorem}\cite[Theorem 5.22]{LEE}\label{th3}
Let $M$ and $N$ be smooth manifolds, and let $\phi \colon M\to N$ be a smooth map with constant rank $k$. Each level set of $\phi $ is a closed embedded submanifold of codimension $k$ in $M$. 
\end{Theorem}

\begin{proposition}\label{prop2}
Let $\epsilon \leq \epsilon _{0}$ and consider $\{\epsilon \}\times L_{p,q}$ as a quotient of $\Se $ under the action \eqref{eq3}. For any $t\in [0,2\pi )$, the inverse image $F_{t}=\Psi ^{-1}(e^{it})\subset (\{\epsilon \}\times L_{p,q})\backslash K$ is a smooth surface.
\end{proposition}
\begin{proof} By Lemma \ref{lemma3}, the map $\Phi $ has no critical points on $\Se \backslash \widetilde{K}$. By Theorem \ref{th1}, the quotient maps $\pi \colon \Se \to L_{p,q}$ and $\mu \colon S^{1}\to S^{1}$ are both smooth covering maps. Thus the induced map $\Psi \colon (\{\epsilon \}\times L_{p,q})\backslash K\to S^{1}$ is a smooth map without critical points and by Theorem \ref{th3}, each level set $F_{t}=\Psi ^{-1}(e^{it})$ is a smooth surface. 
\end{proof}

\begin{lemma} \label{lemma5} The boundary of the closure of $F_{t}$ in $L_{p,q}$ is precisely the link $K$.  
\end{lemma}
\begin{proof} Let $\widetilde{F}_{t}=\pi ^{-1}(F_{t})$ be the preimage of $F_{t}$ in $\Se $. Denote the boundary of the closure of $F_{t}$ in $L_{p,q}$ (respectively closure of $\widetilde{F}_{t}$ in $\Se $) by $\partial F_{t}$ (respectively $\partial \widetilde{F}_{t}$). By \cite[Lemma 6.1]{MIL} we have $\partial \widetilde{F}_{t}=\widetilde{K}$. 

To show that $K\subset \partial F_{t}$, choose an $x\in K$ and let $U\subset L_{p,q}$ be any neighbourhood of $x$. Then $\pi ^{-1}(U)\subset \Se $ is a neighbourhood of the set $\pi ^{-1}(x)\subset \widetilde{K}=\partial \widetilde{F}_{t}$, thus there exists a point $z\in \pi ^{-1}(U)\cap \widetilde{F}_{t}$. It follows that $\pi (z)\in U\cap F_{t}$ and $x\in U\backslash F_{t}$, therefore $x\in \partial F_{t}$.  

To check that $\partial F_{t}\subset K$, let $y\in \partial F_{t}$. Since $F_{t}$ is an open set, $y\notin F_{t}$, therefore $\pi ^{-1}(y)\cap \widetilde{F}_{t}=\emptyset $. Choose an $x\in \pi ^{-1}(y)$ and let $U\subset \Se $ be a neighbourhood of $x$ in $\Se $ such that $\pi |_{U}\colon U\to \pi (U)$ is a diffeomorphism. Then $\pi (U)$ is a neighbourhood of $y$, therefore $\pi (U)\cap F_{t}\neq \emptyset $. Choose an element $z\in \pi (U)\cap F_{t}$ and let $w\in U$ be the unique element for which $\pi (w)=z$. It follows that $w\in U\cap \widetilde{F}_{t}$, therefore $U\cap \widetilde{F}_{t}\neq \emptyset $ and $U\backslash \widetilde{F}_{t}\neq \emptyset $. We have shown that $x\in \partial \widetilde{F}_{t}=\widetilde{K}$, which implies $y=\pi (x)\in K$. 
\end{proof}

\begin{corollary} \label{cor1} Every algebraic link in $L_{p,q}$ is nullhomologous. 
\end{corollary}

For an algebraic link $K$ in $L_{p,q}$ with $r$ components $K_{1},\ldots ,K_{r}$, denote by $\delta _{i}\in H_{1}(L_{p,q})\cong \ZZ _{p}$ the homology class of the component $K_{i}$. By Corollary \ref{cor1}, we have $$\sum _{i=1}^{r}\delta _{i}=0\;.$$ The following result relates the number of components of a link in $L_{p,q}$ and of its lift in $S^{3}$. 

\begin{lemma}\cite[Proposition 2]{MAN}\label{lemma6}
Let $K$ be an $r$-component link in $L_{p,q}$, whose lift is an $\widetilde{r}$-component link in $S^{3}$. Denote by $\delta _{i}\in H_{1}(L_{p,q})\cong \ZZ _p$ the homology class of the $i$-th component of $K$ for $i=1,\ldots ,r$. Then we have $$\sum _{i=1}^{r}\gcd(\delta _{i},p)=\widetilde{r}\;.$$   
\end{lemma}

\begin{corollary}\label{cor3} 
\begin{enumerate}
\item[(a)] If $K$ is an algebraic knot in $L_{p,q}$, then $\widetilde{K}$ is a $p$-component link in $S^{3}$. 
\item[(b)] If $K$ is a 2-component algebraic link in $L_{p,q}$, then $\widetilde{K}$ has $2\gcd (\delta _{1},p)$ components. 
\end{enumerate}
\end{corollary}
\begin{proof}
\begin{enumerate}
\item[(a)] Since $K$ has one component, Corollary \ref{cor1} implies that $\delta _{1}=0$ and by Lemma \ref{lemma6}, its lift has $p$ components. 
\item[(b)] By Corollary \ref{cor1} we have $\delta _{2}=-\delta _{1}$ and thus $\gcd (\delta _{1},p)=\gcd (\delta _{2},p)$. Lemma \ref{lemma6} implies the result. 
\end{enumerate}
\end{proof}

\begin{example}\label{ex6} The zero set of the complex polynomial $f(x,y)=x^{8}+y^{2}$ intersects the 3-sphere in the 2-component torus link $T(8,2)$. In Example \ref{ex2} we have shown that $T(8,2)$ is the lift of a 2-component algebraic link in $L_{3,1}$, whose components represent homology classes $\pm 1\in H_{1}(L_{3,1})$.   

The polynomial $f(x,y)$ is also $(2,1)$-invariant, thus $T(8,2)$ is the lift of an algebraic link in $L_{2,1}$. Indeed: consider the knot $3_7$ in $L_{2,1}$, that is given by diagrams in Figure \ref{fig4}. In the lens space $L_{2,1}$, the knot $3_7$ is equivalent to the knot $2_1$ (see Figure \ref{fig6}) by \cite[Appendix B]{GA}. Applying Proposition \ref{prop6}, we can see that the lift of $3_7$ in the 3-sphere is equivalent to the torus link $T(8,2)$.  We have thus shown that $T(8,2)$ is at the same time the lift of an algebraic knot in $L_{2,1}$, and the lift of a 2-component algebraic link in $L_{3,1}$. 
\begin{figure}[h!]
\labellist
\normalsize \hair 2pt
\endlabellist
\begin{center}
\includegraphics[scale=0.25]{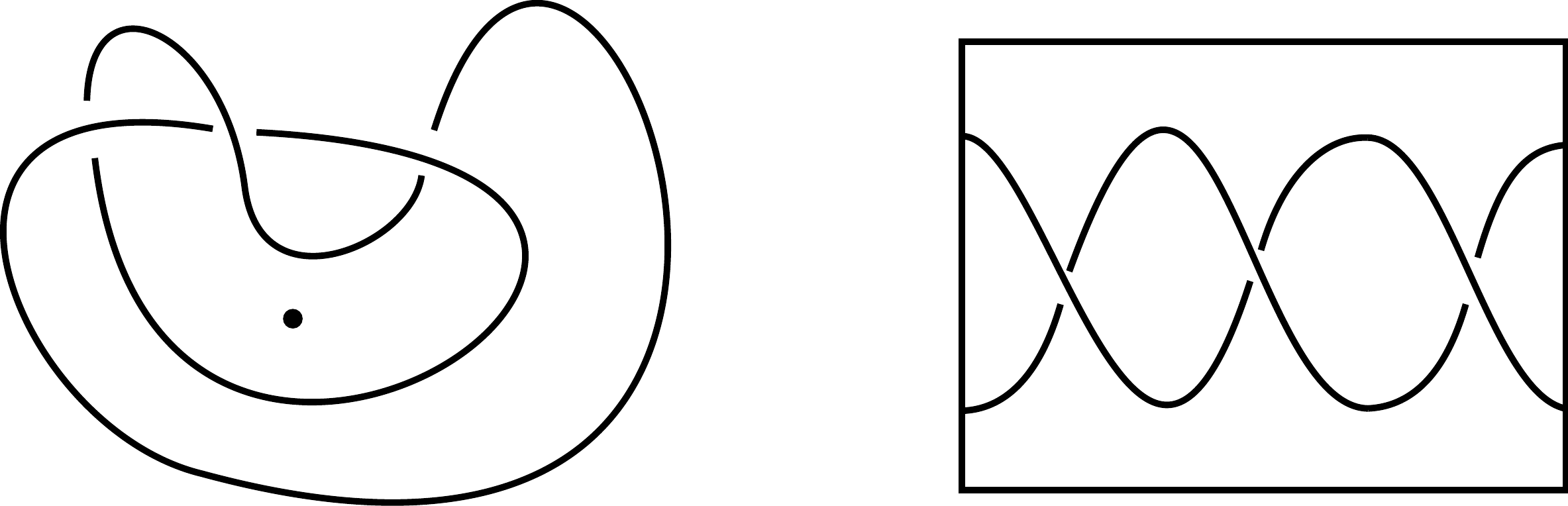}
\caption{The algebraic knot $3_7$ in $L_{2,1}$ from Example \ref{ex6} (punctured disk diagram on the left and band diagram on the right)}
\label{fig4}
\end{center}
\end{figure}
\end{example}

Recall the commutative diagram of smooth maps 
\begin{displaymath}
\xymatrix{\Se \backslash \widetilde{K} \ar@{->}[r]^{\pi } \ar@{->}[d]_{\Phi } &  L_{p,q}\backslash K  \ar@{->}[d]^{\Psi }\\
S^{1} \ar@{->}[r]^{\mu } & S^{1}}
\end{displaymath}
Firstly, we combine the maps $\mu $ and $\Phi $ to obtain:

\begin{lemma}\label{lemma4} Let $\epsilon _{0}$ be as in Lemma \ref{lemma3}, and let $\epsilon \leq \epsilon _{0}$. The map $\mu \circ \Phi \colon \Se \backslash \widetilde{K}\to S^{1}$ is the projection of a smooth fiber bundle, whose fiber consists of $\overline{p}$ disjoint fibers of the fiber bundle $\Phi \colon \Se \backslash \widetilde{K}\to S^{1}$.  
\end{lemma}
\begin{proof} Choose a point $e^{it}\in S^{1}$, then $\mu ^{-1}(e^{it})=\{e^{i(\frac{t+2\pi j}{\overline{p}})}|\, j=1,\ldots ,\overline{p}\}$. Since $\mu $ is a covering map, there exists a neighbourhood $U\subset S^{1}$ of $e^{it}$, such that $\mu ^{-1}(U)=U_{1}\sqcup \ldots \sqcup U_{\overline{p}}$, and $\mu |_{U_{j}}$ is a diffeomorphism for $j=1,\ldots ,\overline{p}$. We have $(\mu \circ \Phi )^{-1}(U)=\Phi ^{-1}(U_{1})\sqcup \ldots \sqcup \Phi ^{-1}(U_{\overline{p}})$. By Theorem \ref{th2}, the map $\Phi $ is the projection of a smooth fiber bundle. Denote by $\widetilde{F}_{j}=\Phi ^{-1}\left (e^{i(\frac{t+2\pi j}{\overline{p}}}\right )$ its fibers at the points of $\mu ^{-1}(e^{it})$; then there exist diffeomorphisms $h_{j}\colon U_{j}\times \widetilde{F}_{j}\to \Phi ^{-1}(U_{j})$ for $j=1\ldots ,\overline{p}$. Define a map $h\colon U\times \left (\bigsqcup _{j=1}^{\overline{p}}\widetilde{F}_{j}\right )\to \bigsqcup _{j=1}^{\overline{p}}\Phi ^{-1}(U_{j})$ by $$h(u,z)=h_{j}\left ((\mu |_{U_{j}})^{-1}(u),z\right )\textrm{   for }z\in \widetilde{F}_{j}\;.$$ Then $h$ is a diffeomorphism with inverse $h^{-1}(w)=(\mu ,\textrm{id})\circ h_{j}^{-1}(w)$ for $w\in \Phi ^{-1}(U_{j})$.  
\end{proof}

Now we are prepared to prove the following. 

\begin{Theorem} \label{th4} For $\epsilon \leq \epsilon _{0}$, the space $\left (\{\epsilon \}\times L_{p,q}\right )\backslash K$ is a smooth fiber bundle over $S^{1}$ with projection mapping $\Psi $. 
\end{Theorem}
\begin{proof} By Theorem \ref{th1}, the map $\pi \colon \Se \backslash \widetilde{K}\to \left (\{\epsilon \}\times L_{p,q}\right )\backslash K$  is a smooth covering map. Denote by $Z_{1},Z_{2}\ldots ,Z_{p}\subset \Se \backslash \widetilde{K}$ the distinct leaves of this covering. Then $\pi |_{Z_{1}}\colon Z_{1}\to \left (\{\epsilon \}\times L_{p,q}\right )\backslash K$ is a diffeomorphism. 

Choose a point $e^{it}\in S^{1}$. Denote by $\widetilde{F}_{t}=(\mu \circ \Phi )^{-1}(e^{it})$ and $F_{t}=\Psi ^{-1}(e^{it})$ its fibers. By Lemma \ref{lemma4}, the map $\mu \circ \Phi \colon \Se \backslash \widetilde{K}\to S^{1}$ is the projection of a smooth fiber bundle, so there exists a neighbourhood $U\subset S^{1}$ of $e^{it}$ and a diffeomorphism $h\colon U\times \widetilde{F}_{t}\to (\mu \circ \Phi )^{-1}(U)$. Define a map $g\colon U\times F_{t}\to \psi ^{-1}(U)$ by $$g(u,\pi (x,y))=\pi \left (h\left (u,(\pi |_{Z_{1}})^{-1}(\pi (x,y))\right )\right )\;.$$ Since $\pi |_{Z_{1}}$ is a diffeomorphism, $g$ is a well-defined smooth map. Moreover, $g$ is a diffeomorphism with inverse $g^{-1}(\pi (x,y))=(\textrm{id},\pi )\circ h^{-1}\left ((\pi |_{Z_{1}})^{-1}(\pi (x,y))\right )$. 
\end{proof}

\begin{corollary} \label{cor2} Every algebraic link in a lens space is fibered. 
\end{corollary}

The relationship between the fibre bundles $\Se \backslash \widetilde{K}$ and $L_{p,q}\backslash K$ can also be described in terms of monodromy. Let $g\colon \widetilde{F}\to \widetilde{F}$ denote the monodromy map of the fiber bundle $\mu \circ \Phi \colon \Se \backslash \widetilde{K}\to S^{1}$, thus $\Se \backslash \widetilde{K}$ is diffeomorphic to $(\widetilde{F}\times I)/_{(x,0)\sim (g(x),1)}$. Denote by $F$ the fiber of the fibre bundle $\Psi \colon L_{p,q}\backslash K\to S^{1}$. The surface $\widetilde{F}$ may be decomposed into $p$ fundamental parts $\widetilde{F}_{1},\ldots ,\widetilde{F}_{p}$, so that $\pi |_{\widetilde{F}_{i}}\colon \widetilde{F}_{i}\to F$ is a diffeomorphism for each $i$. Then the monodromy map of the fiber bundle $\Psi \colon L_{p,q}\backslash K\to S^{1}$ is given by $\pi \circ g\circ (\pi |_{\widetilde{F}_{1}})^{-1}\colon F\to F$. \\

Recall that for a fibered link in $S^{3}$, a fiber surface is the unique surface that realizes the minimal Seifert genus of the link \cite{KO}, \cite{ST}. This fact generalizes to links in an arbitrary 3-manifold:

\begin{proposition}\cite[Proposition 4.1]{EN} \label{prop3}
If $L$ is a fibered link in a 3-manifold $M$, then any fiber is a minimal Seifert surface and conversely, any minimal Seifert surface is isotopic to a fiber.   
\end{proposition}

\begin{corollary}\label{cor4}
Let $K$ be an algebraic link in $L_{p,q}$, and let $\Psi \colon L_{p,q}\backslash K\to S^{1}$ be the corres\-pon\-ding fibration. Then the closure of each fiber $F_{t}=\Psi ^{-1}(e^{it})$ in $L_{p,q}$ is a Seifert surface of minimal genus for $K$. 
\end{corollary}
\begin{proof} It follows immediately from Theorem \ref{th4}, Proposition \ref{prop3} and Lemma \ref{lemma5}. 
\end{proof}

The genus of the fiber $F_{t}$ is related to the genus of $\pi ^{-1}(F_{t})=\widetilde{F}_{t}$. In case of an algebraic knot in a lens space, we obtain:

\begin{proposition} \label{prop4} Let $K$ be an algebraic knot in $L_{p,q}$, and let $\widetilde{K}$ be its lift in $S^{3}$. Denote by $g$ (respectively $\widetilde{g}$) the Seifert genus of $K$ (respectively $\widetilde{K}$). Then we have $$g=\frac{2\widetilde{g}+p+\gcd (p,k)-2}{2\gcd (p,k)}\;.$$
\end{proposition}
\begin{proof} Denote by $\widetilde{F}_{t}=(\mu \circ \Phi )^{-1}(e^{it})$ (respectively $F_{t}=\Psi ^{-1}(e^{it})$) the fibers that represent interiors of the Seifert surfaces for $\widetilde{K}$ (respectively $K$) . The covering $\pi \colon \widetilde{F}_{t}\to F_{t}$ is cyclic of order $p$, so $\chi (\widetilde{F}_{t})=p\cdot \chi (F_{t})$. By Lemma \ref{lemma4}, $\widetilde{F}_{t}$ is a disjoint union of $\overline{p}$ diffeomorphic copies of the fiber $\Phi ^{-1}(e^{it})$, which in turn is connected by \cite[Corollary 6.3]{MIL}. The boundary of $\widetilde{F}_{t}$ in $\Se $ has $p$ components by Corollary \ref{cor3}. Since $\overline{p}=\frac{p}{\gcd (p,k)}$, we may compute
\begin{xalignat*}{1}
& \overline{p}(2-2\widetilde{g}-p)=p(1-2g)\\
& \frac{p}{\gcd (p,k)}(2-2\widetilde{g}-p)=p(1-2g)\\
& 2-2\widetilde{g}-p=\gcd (p,k)(1-2g)
\end{xalignat*} and solving for $g$ yields the expression above. 
\end{proof}

\section{Examples}\label{sec3}

In this Section, we find some examples of algebraic knots and links in lens spaces. We concentrate on the case when the lift of the algebraic knot/link in the 3-sphere is a torus link. 

\begin{proposition} \label{prop5} Let $a,b,p$ and $q$ be integers, with $p$ and $q$ relatively prime. The torus link $T(a,b)$ is the lift of an algebraic link in $L_{p,q}$ $\Leftrightarrow $ $a\equiv qb\textrm{ (mod p)}$.  
\end{proposition} 
\begin{proof} The torus link $T(a,b)$ is the link of the singularity $(0,0)\in \CC ^{2}$ of the complex polynomial $f(x,y)=x^{a}+y^{b}$. Let $\zeta \in \CC $ be a primitive $p$-th root of unity. Then we have $f(\zeta x,\zeta ^{q}y)=\zeta ^{a}x^{a}+\zeta ^{qb}y^{b}$ and $\zeta ^{k}f(x,y)=\zeta ^{k}(x^{a}+y^{b})$. The polynomial $f$ is $(p,q)$-invariant if and only if there exists an integer $0\leq k<p$, such that the equality $$(\zeta ^{a}-\zeta ^{k})x^{a}+(\zeta ^{qb}-\zeta ^{k})y^{b}=0$$ holds for every $(x,y)\in \CC ^{2}$, which is equivalent to $a\equiv qb\textrm{ (mod p)}$.   
\end{proof}

\begin{corollary} \label{cor5} Let $a,b,p$ and $q$ be integers, with $p$ and $q$ relatively prime. The torus link $T(a,b)$ is the lift of an algebraic knot in $L_{p,q}$ if and only if $\gcd (a,b)=p$. In this case, the Seifert genus of the knot equals $g=\frac{\widetilde{g}+p-1}{p}$, where $\widetilde{g}$ is the Seifert genus of $T(a,b)$. 
\end{corollary}
\begin{proof} Let $T(a,b)$ be the lift of an algebraic knot in $L_{p,q}$. Since the torus link $T(a,b)$ has $\gcd (a,b)$ components, Corollary \ref{cor3} implies that $\gcd (a,b)=p$. 

Conversely, if $\gcd (a,b)=p$, then $a\equiv qb\textrm{ (mod $p$)}$ for every $q$, so $T(a,b)$ is the lift of an algebraic link in $L_{p,q}$ by Proposition \ref{prop5}. The algebraic set $x^{a}+y^{b}=0$ consists of $p$ nonsingular branches, and the action of $G_{p,q}$ induces a cyclic permutation of these branches. Under the quotient map $\pi \colon \Se \to L_{p,q}$, all the $p$ components of $T(a,b)$ identify. Therefore the algebraic link in $L_{p,q}$ is actually a knot. 

To prove the last statement of the Corollary, let $a=a_{1}p$ and $b=b_{1}p$ with $\gcd (a_{1},b_{1})=1$. Since $f(x,y)=x^{a_{1}p}+y^{b_{1}p}=f(\zeta x,\zeta ^{q}y)$ for every $(x,y)\in \CC ^{2}$, we have $k=0$ and $\gcd (p,k)=p$. Using Proposition \ref{prop4}, we obtain $$g=\frac{2\widetilde{g}+p+p-2}{2p}=\frac{\widetilde{g}+p-1}{p}\;.$$ 
\end{proof}

\begin{example}\label{ex3} It follows from Corollary \ref{cor5} that the torus link $T(9,3)$ is the lift of an algebraic knot in $L(3,q)$ for $q=1,2$. The link $T(9,3)$ is the closure of a braid on 3 strands, that is given by the braid word $(\sigma _{2}\sigma _{1})^{9}$. We use Proposition \ref{prop6} with $n=3$ and note that the braid relation $\sigma _{i}\sigma _{i+1}\sigma _{i}=\sigma _{i+1}\sigma _{i}\sigma _{i+1}$ implies $(\Delta _{3})^{2}=\sigma _{2}\sigma _{1}\sigma _{1}\sigma _{2}\sigma _{1}\sigma _{1}=(\sigma _{2}\sigma _{1})^{3}$ and $(\Delta _{3})^{4}=(\sigma _{2}\sigma _{1})^{6}$. It follows that $T(9,3)$ is the lift of a knot $K_{1}$ in $L_{3,1}$, whose band diagram corresponds to the braid word $(\sigma _{2}\sigma _{1})^{2}$. $T(9,3)$ is also the lift of the knot $K_{2}$ in $L_{3,2}$, whose band diagram corresponds to the braid word $\sigma _{2}\sigma _{1}$ (see Figure \ref{fig5}). It follows from the Corollary \ref{cor2} that $K_{1}$ and $K_{2}$ are both fibered knots. 
\begin{figure}[h!]
\labellist
\normalsize \hair 2pt
\endlabellist
\begin{center}
\includegraphics[scale=0.25]{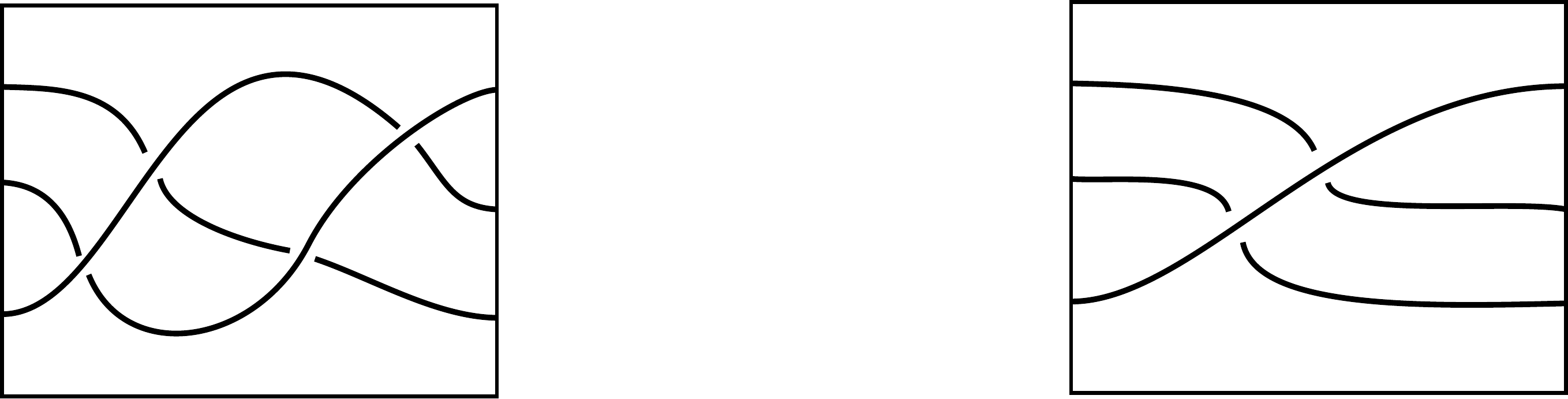}
\caption{The algebraic knots from Example \ref{ex3}: $K_1$ in $L_{3,1}$ (left) and $K_{2}$ in $L_{3,2}$ (right)}
\label{fig5}
\end{center}
\end{figure}
\end{example}

\begin{example}\label{ex4}
By Corollary \ref{cor5}, the torus link $T(3,3)$ is the lift of a knot in $L_{3,1}$. Figure \ref{fig6} shows the knot $2_1$ from the knot atlas in \cite{GA}. Figure \ref{fig7} shows that the lift in $S^{3}$ of the knot $2_1$ in $L_{3,1}$ is equivalent to the link $T(3,3)$. It follows from the Corollary \ref{cor2} that $2_1$ is a fibered knot in $L_{3,1}$. 

\begin{figure}[h!]
\labellist
\normalsize \hair 2pt
\endlabellist
\begin{center}
\includegraphics[scale=0.25]{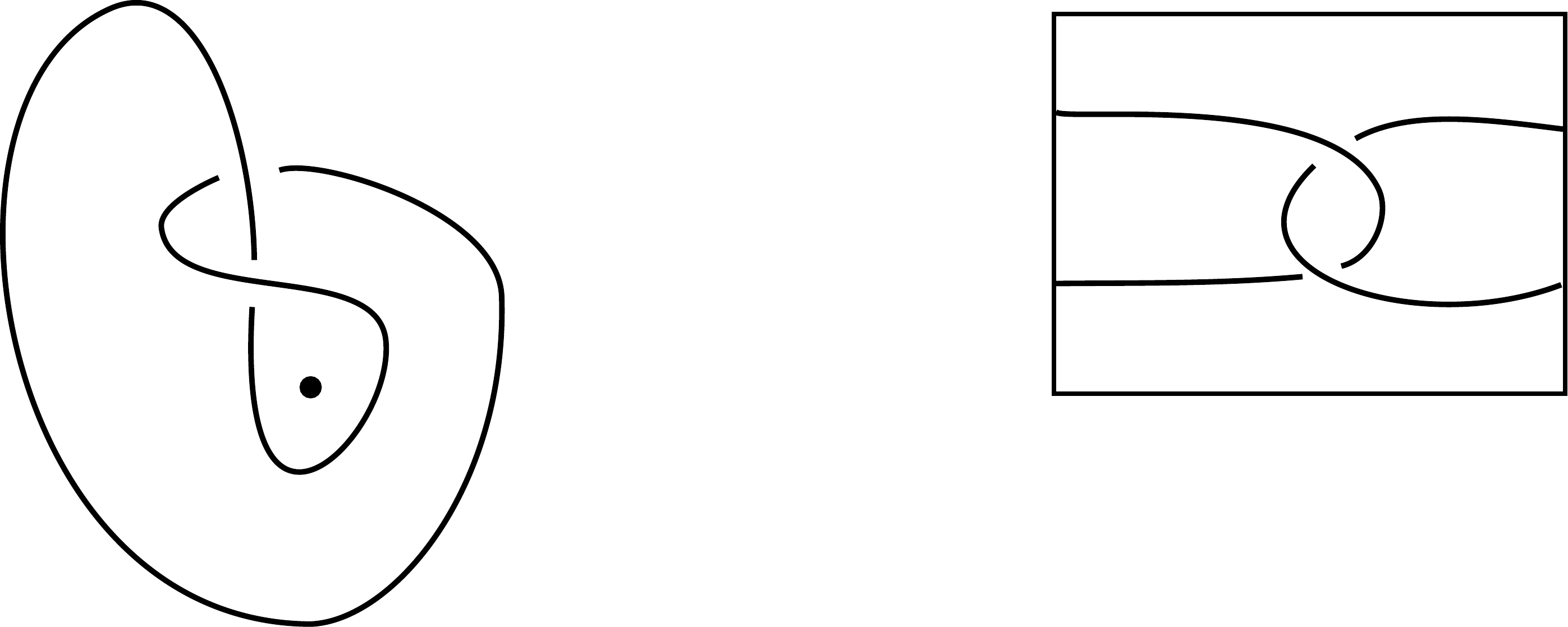}
\caption{The knot $2_1$ from the knot atlas in \cite{GA} (punctured disk diagram on the left and band diagram on the right)}
\label{fig6}
\end{center}
\end{figure}

\begin{figure}[h!]
\labellist
\normalsize \hair 2pt
\pinlabel $\sim $ at 1070 230
\pinlabel $\sim $ at 1600 230
\endlabellist
\begin{center}
\includegraphics[scale=0.25]{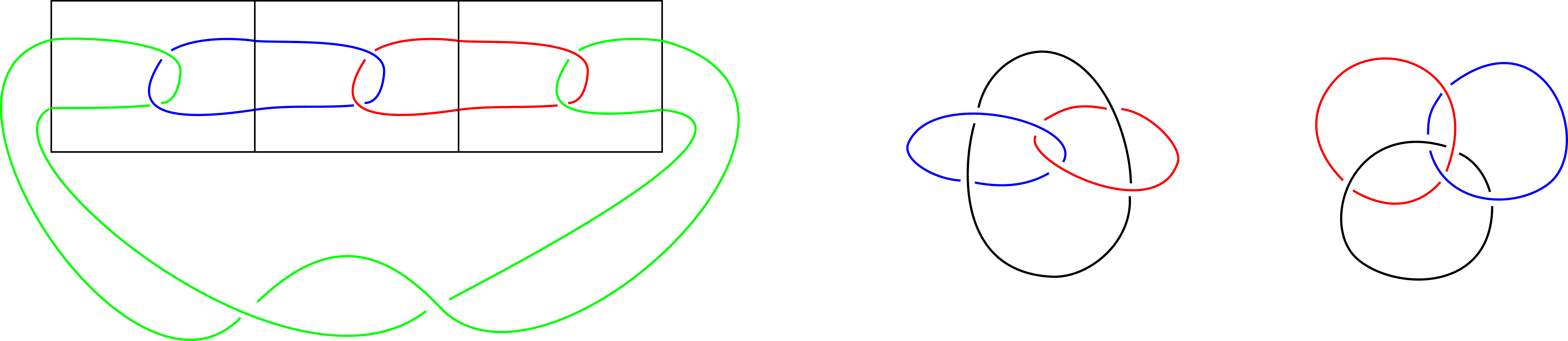}
\caption{The lift in $S^{3}$ of the knot $2_1$ in $L_{3,1}$, see Example \ref{ex4}}
\label{fig7}
\end{center}
\end{figure}
\end{example}

\begin{example}\label{ex5}
By Corollary \ref{cor5}, the torus link $T(4,2)$ (also called the Solomon's knot) is the lift of a knot in $L_{2,1}$. Indeed: Figure \ref{fig8} shows the knot $3_8$ from the knot atlas in \cite{GA}. We obtain the diagram of its lift in $S^{3}$ by Proposition \ref{prop6}. Figure \ref{fig9} shows that the lift is equivalent to the Solomon's knot. It follows by Corollary \ref{cor2} that $3_8$ is a fibered knot in $L_{2,1}$. 

\begin{figure}[h!]
\labellist
\normalsize \hair 2pt
\endlabellist
\begin{center}
\includegraphics[scale=0.25]{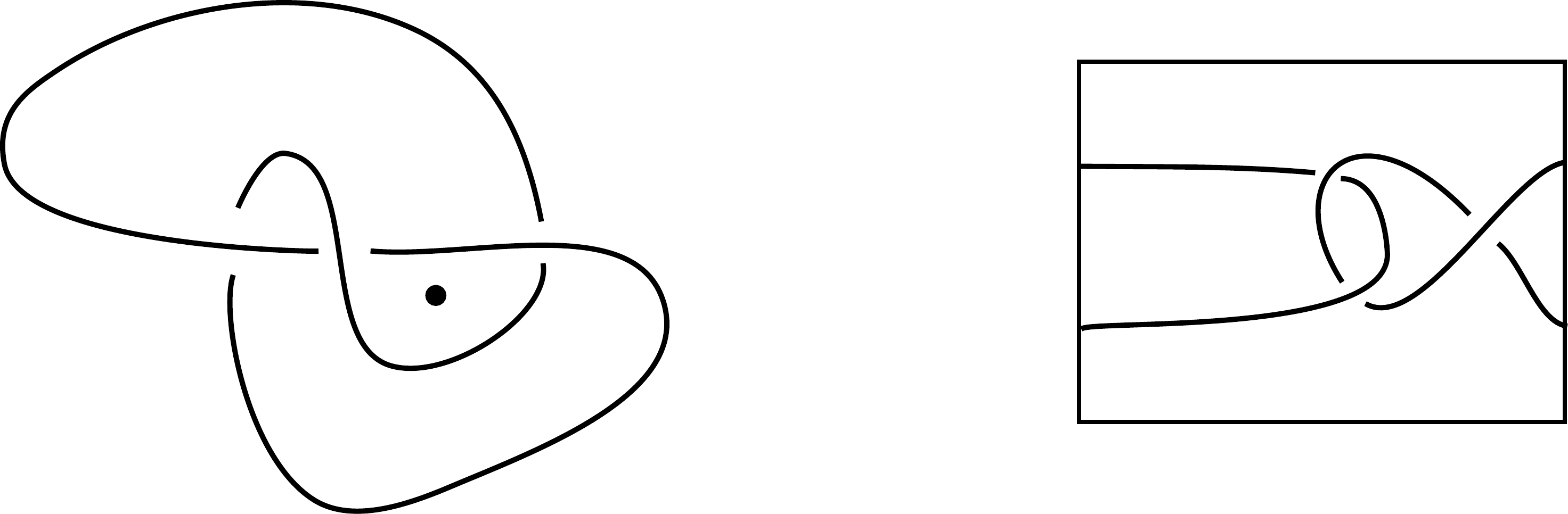}
\caption{The knot $3_8$ from the knot atlas in \cite{GA}}
\label{fig8}
\end{center}
\end{figure}

\begin{figure}[h!]
\labellist
\normalsize \hair 2pt
\pinlabel $\sim $ at 800 230
\pinlabel $\sim $ at 1325 220
\endlabellist
\begin{center}
\includegraphics[scale=0.25]{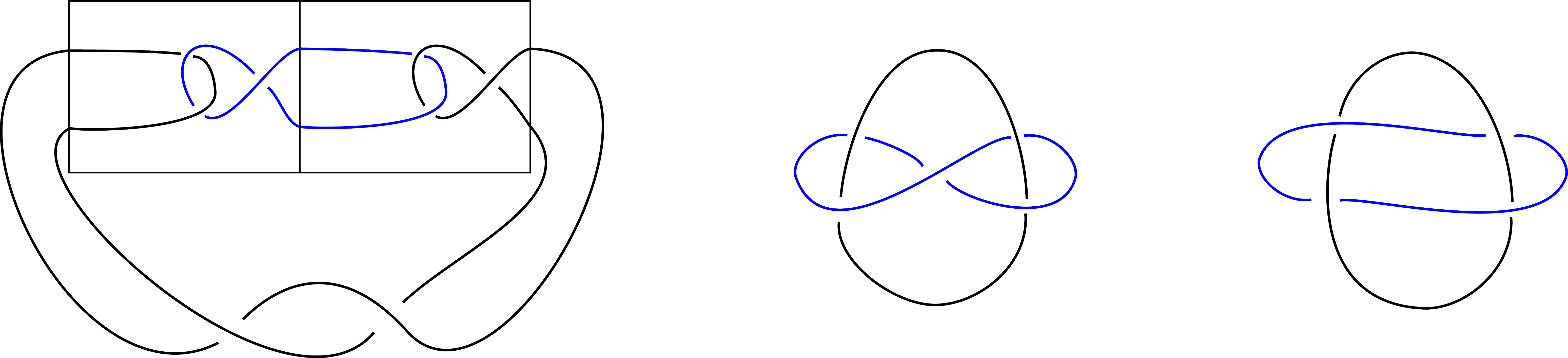}
\caption{The lift in $S^{3}$ of the knot $3_8$ in $L_{2,1}$, see Example \ref{ex5}}
\label{fig9}
\end{center}
\end{figure}
\end{example}

More generally, an ample supply of algebraic knots in $L_{p,q}$ is guaranteed by the following construction:

\begin{lemma} \label{lemma7} Let $f(x,y)$ be an irreducible complex polynomial with $f(0,0)=0$. For any pair of positive integers $p$ and $q$ with $gcd(p,q)=1$, the polynomial $f_{p}(x,y)=f(x^{p},y^{p})$ defines an algebraic knot in $L_{p,q}$.  
\end{lemma}
\begin{proof} The polynomial $f_{p}$ is obviously $(p,q)$-invariant. Since $f$ is irreducible, the algebraic curve $\{(x,y)\in \CC ^{2}\, |\, f_{p}(x,y)=0\}$ consists of $p$ branches. The action of $G_{p,q}$ induces a cyclic permutation of these branches. The link of the singularity at the origin is therefore a $p$-component link in $S^{3}$, whose quotient in $L_{p,q}$ is an algebraic knot. 
\end{proof}

\section{Open problems}\label{sec4}
\begin{enumerate}
\item Torus links represent the simplest class of classical algebraic links. Proposition \ref{prop5} and Corollary \ref{cor5} classify torus knots and links that are lifts of algebraic links in $L_{p,q}$. It would be interesting to find a topological characterization of the algebraic links in $L_{p,q}$ that lift to torus links.
\item More generally, can we find a topological characterization of \textbf{all} algebraic links in lens spaces? 
\item For a $(p,q)$-invariant polynomial $f$, one might investigate the complex algebraic curve $V=f^{-1}(0)\subset (0,\infty )\times S^{3}$ and its quotient in $\pi (V)\subset (0,\infty )\times L_{p,q}$. How do techniques that we use to study plane algebraic curves, translate to the study of $\pi (V)$? Applying them, we might gain new insight about the topology of the quotient curve. 
\item It is known that a smooth fibration $\Psi \colon M\backslash K\to S^{1}$ of a closed, orientable 3-manifold $M$ defines an \textbf{open book decomposition} of $M$. By \cite{GI}, there is a one-to-one corres\-pon\-dence between isotopy classes of oriented contact structures on $M$ and open book decompositions up to positive stabilizations. By Theorem \ref{th4}, every algebraic knot $K$ in $L_{p,q}$ defines a fibration $\Psi \colon L_{p,q}\backslash K\to S^{1}$. Can we determine which contact structures on $L_{p,q}$ correspond to these fibrations? Does every contact structure on $L_{p,q}$ correspond to the fibration of some algebraic knot in $L_{p,q}$? Are there algebraic knots/links that determine the same contact structure? 
\item The smooth 4-genus of a classical algebraic link $K$ is known to be realized by the algebraic curve whose boundary is $K$. Is it possible to generalize this result to the algebraic links in lens spaces? The first problem we meet is a meaningful definition of the smooth 4-genus of lens space knots. As opposed to the 3-sphere, a typical lens space is not the boundary of the 4-ball (which may be conveniently completed to the complex projective plane). Here we need to move from the complex to the more general symplectic setting. The symplectic Thom conjecture, proven by Ozsv\'{a}th and Szab\'{o}, states that a symplectic surface in a symplectic 4-manifold is genus-minimizing in its homology class. The symplectic fillings of lens spaces have been investigated by Lisca \cite{LIS}, McDuff \cite{MCD} and recently by Etnyre and Roy \cite{ETN}. Suppose that $W_{p,q}$ is a symplectic filling of the lens space $L_{p,q}$, equipped with the standard contact structure. Define the smooth 4-genus of a knot $K$ in $L_{p,q}$ to be the minimum genus of a smooth surface embedded in $W_{p,q}$, whose boundary is $K$. Before stating this as a definition, of course, one should check whether this minimum genus depends on the choice of the symplectic filling $W_{p,q}$. If it does not, then we could explore further to see whether the symplectic Thom conjecture implies results about the smooth 4-genus of algebraic knots in $L_{p,q}$. 
\end{enumerate}

\section*{Acknowledgements}

The author was supported by the Slovenian Research Agency grant N1-0083.

\end{document}